\documentclass[12pt,a4paper]{article}

\usepackage{amsmath,amstext,amssymb,amscd,euscript}
 \usepackage{graphicx}

\usepackage[russian,english]{babel}
\usepackage[cp1251]{inputenc}

\newcommand{\Xcomment}[1]{}
\newtheorem{theorem}{Theorem}[section]
\newtheorem{lemma}[theorem]{Lemma}

\newenvironment{proof}{\noindent{\bf Proof}\/}%
{\hfill$\qed$\medskip}

\def\qed{\Box}

\makeatletter \@addtoreset{equation}{section} \makeatother

\newenvironment{numitem1}{\refstepcounter{equation}\begin{enumerate}%
\item[(\thesection.\arabic{equation})]}{\end{enumerate}}

\newcommand{\refeq}[1]{(\ref{eq:#1})}  

 \makeatletter
\renewcommand{\section}{\@startsection{section}{1}{0pt}%
{-3.5ex plus -1ex minus -.2ex}{2.3ex plus .2ex}%
{\normalfont\Large}}
 \makeatother

 \makeatletter
\renewcommand{\subsection}{\@startsection{subsection}{2}{0pt}%
{-3.0ex plus -1ex minus -.2ex}{1.5ex plus .2ex}%
{\normalfont\normalsize\bf}}
 \makeatother

 \newcommand{\SEC}[1]{\ref{sec:#1}}  
\newcommand{\SSEC}[1]{\ref{ssec:#1}}  

\def\Rset{{\mathbb R}}

\def\Zset{{\mathbb Z}}

\def\Ascr{{\cal A}}
\def\Bscr{{\cal B}}
\def\Cscr{{\cal C}}

\def\Escr{{\cal E}}
\def\Fscr{{\cal F}}
\def\Gscr{{\cal G}}

\def\Mscr{\EuScript{M}}

\def\Rscr{{\cal R}}
\def\Sscr{{\cal S}}
\def\Tscr{{\cal T}}

\def\Vscr{{\cal V}}
\def\Wscr{{\cal W}}
\def\Xscr{{\cal X}}
\def\Yscr{{\cal Y}}
\def\Zscr{{\cal Z}}

\def\hat{\widehat}
\def\bar{\overline}
\def\eps{\varepsilon}

\def\xmin{x^{\rm min}}
\def\xmax{x^{\rm max}}

\def\bmax{b^{\rm max}}

\def\Sscrpr{\Sscr^{\rm pr}}
\def\precpr{\prec^{\rm pr}}
\def\phipr{\phi^{\rm pr}}

\def\onebf{{\bf 1}}

\def\trleft{\vartriangleleft}
\def\trlefteq{\trianglelefteq}

\begin{document}
\parskip=2pt

\title{On splitting properties of the stability problem with integer choice functions }
\author{Alexander V.~Karzanov
\thanks{Central Institute of Economics and Mathematics of the RAS, 47, Nakhimovskii Prospect, 117418 Moscow, Russia; email: akarzanov7@gmail.com.}
}
\date{}

 \maketitle
\vspace{-1.0cm}
\begin{abstract}
We consider the integer version of Alkan--Gale's model on stability in a two-sided market, called the \emph{stable generalized allocation} one. It is given by a triple $(G,b,C)$, where $G=(V,E)$ is a finite bipartite graph with nonnegative integer \emph{capacities} $b(e)\in\Zset_+$ of edges $e\in E$, and for each vertex (``agent'') $v\in V$, the preferences on the set $E_v$ of its incident edges depend on a \emph{choice function} $C_v$. The latter acts on the set of vectors in $\Zset_+^{E_v}$ bounded by the capacities and obeys the standard axioms of \emph{substitutability} and \emph{size monotonicity}. Alkan--Gale's prominent theorem implies that the stability problem in this case always has a stable solution $x\in\Zset_+^E$ and, moreover, the set $\Sscr_{G,b,C}$ of these solutions (``stable generalized allocations'') forms a distributive lattice.

However, this lattice is rather intricate to construct and work with, and we wonder whether it can be represented via a ``simpler'' stability model. Answering this issue, we arrange a sort of splitting techniques to embed $\Sscr_{G,b,C}$, as a sublattice, in the lattice of stable matchings and, more compactly, in the lattice of stable allocations (as in Baiou--Balinski's stability model). This generalizes Fleiner's result on a detachment in the special case with all-unit capacities.

\bigskip

\end{abstract}


\section{Introduction}  \label{sec:intr}

The appearance of the famous work on stable marriages due to Gale and Shapley~\cite{GS} has inspired a flow of subsequent valuable researches on stable assignments in two-sided markets. One popular stability model on this way was introduced and studied by Baiou and Balinski~\cite{BB}, known as the \emph{stable allocation problem} (briefly \emph{SAP}). In its setting (in the real-valued version), there is given a finite bipartite graph $G=(V,E)$ in which the edges $e\in E$ have nonnegative real upper bounds, or \emph{capacities} $b(e)\in\Rset_+$, the vertices $v\in V$ have \emph{quotas} $q(v)\in\Rset_+$, and each vertex (``agent'') $v$ is endowed with a linear order (of ``preferences'') on the set $E_v$ of its incident edges, serving to compare possible ``allocations'' within $E_v$ in a natural way.

Alkan and Gale~\cite{AG} proposed a far generalization of SAP and other known bipartite stability problems (such as of types ``one-to-many'', ``many-to-many'', and the like). It deals with a bipartite graph $G=(V,E)$ and a vector $b\in\Rset_+^E$ of edge capacities for which ``preferences'' at each vertex $v\in V$ depend on a \emph{choice function} (CF) $C_v$ that acts on the vectors in a closed subset of the real box $\{z\in\Rset_+^{E_v}\colon z(e)\le b(e)\;\forall e\in E_v\}$. Each $C_v$ is assumed to obey the well-known axioms of \emph{substitutability} and \emph{size monotonicity}. The main results in~\cite{AG} are that this model guarantees the existence of a stable assignment and that the set of these forms a distributive lattice. In addition, the important property of \emph{unisizeness} was established, saying that for each vertex $v\in V$, the value $\sum(x(e)\colon e\in E_v)$ is the same for all stable assignments $x\in\Rset_+^E$.

An extensive study of the integer version of Alkan--Gale's model was recently conducted in~\cite{karz}. In that paper one assumes that the capacities $b$ are integral and each choice function $C_v$ acts on the entire integer box $\Bscr_v=\{z\in\Zset_+^{E_v}\colon z(e)\le b(e)\;\forall e\in E_v\}$; we will refer to this version as the \emph{integer Alkan--Gale's problem}, or the \emph{stable generalized allocation problem} (briefly \emph{SGAP}). Stable assignments in this model are called stable \emph{generalized allocations}, or stable \emph{g-allocations} for short, and the set of these is denoted by $\Sscr=\Sscr_{G,b,C}$. Following~\cite{AG}, $\Sscr$ admits a natural partial order $\prec$, turning it into a distributive lattice. Also for each vertex $v\in V$, we denote by $\theta(v)$ the value $\sum(x(e)\colon e\in E_v)$ where $x\in \Sscr$ (by the unisizeness property, this value does not depend on $x$). Two groups of main results on SGAP were established in~\cite{karz}.

The first group concerns local properties of the lattice $(\Sscr,\prec)$. One shows that for any stable $x$, each of its immediate successors $x'$ in this lattice can be obtained by an augmentation w.r.t. a certain edge-simple cycle $R$ in $G$, called a \emph{rotation} applicable to $x$. These rotations for $x$ are pairwise edge-disjoint and can be found in strongly polynomial time. (Hereinafter, speaking of computational complexity, we assume that each choice function $C_v$ is given via an \emph{oracle} that, being asked of a vector $z\in\Bscr_v$, outputs the vector $C_v(z)$. We liberally assume that this operation (an \emph{oracle call}) takes $O(1)$ time.) A rotation $R$ can be applied, step by step, several times, giving a sequence $x=x_0,x_1,\ldots,x_k$
of stable g-allocations, where each $x_i$ immediately succeeds $x_{i-1}$; in this case we say that $R$ can be applied to $x$ with the \emph{feasible weight} $k$. One shows that to find the maximal feasible weight for $(x,R)$, denoted by $\tau_R(x)$, takes a pseudo polynomial time, namely, $O(\bmax|E|^2)$, where $\bmax$ denotes $\max(b(e)\colon e\in E)$.

The second group involves more global properties, concerning an explicit poset representation of the lattice $(\Sscr,\prec)$. On this way, one shows that there exists, and can be constructed in pseudo polynomial time, a poset $(\Pi,\lessdot)$ with the following properties: (a) each element of $\Pi$ is a pair $(R,\tau)$, where $R$ is a rotation applicable to some stable $x\in\Sscr$ with the maximal weight $\tau$ (i.e. $\tau=\tau_R(x)$); and (b) $(\Sscr,\prec)$ is isomorphic to the lattice of closed functions for $(\Pi,\lessdot)$. Here a function $\lambda:\Pi\to\Zset_+$ is called \emph{closed} if it is bounded by the weights $\tau$ (i.e. $\lambda(R,\tau)\le \tau$ for each $(R,\tau)\in\Pi$), and for any $(R,\tau),(R',\tau')\in\Pi$ satisfying $(R,\tau)\lessdot(R',\tau')$ and $\lambda(R',\tau')>0$, there holds $\lambda(R,\tau)=\tau$. An important fact is that all elements of $\Pi$ can be found by constructing an arbitrary sequence $\Tscr=(x_0,x_1,\ldots,x_N)$ such that $x_0$ and $x_N$ are the minimal and maximal element in $(\Sscr,\prec)$, respectively, and for each $i=1,\ldots,N$, $x_i$ is obtained from $x_{i-1}$ by applying a rotation $R_i$ with the maximal feasible weight $\tau_i$; such a $\Tscr$ is called a \emph{full route} in~\cite{karz}. Then $\Pi$ is just the family $\{(R_i,\tau_i), i=1,\ldots,N\}$ (where $(R_i,\tau_i)=(R_j,\tau_j)$ for some $i\ne j$ is possible). So the construction of $\Pi$ looks rather straightforward. As to the order relation $\lessdot$, the task of computing it attracts more sophisticated tools (so as to construct the poset $(\Pi,\lessdot)$ in pseudo polynomial time).

In light of this, it is tempting to try to represent the lattice $(\Sscr,\prec)$ for SGAP in an easier way, by attracting a ``simpler'' stability model. In this paper we resolve this issue by working out some splitting techniques. Two constructions will be applied. Let $W$ (``workers'') and $F$ (``firms'') be the vertex parts (color classes) in $G$.

In the first construction, we split each vertex $v\in V$ into $\max\{\theta(v),1\}$ copies, and split each edge $e\in E$ with endvertices $u$ and $v$ into $b(e)$ edges connecting some copies of $u$ with some copies of $v$ (admitting multiple edges), assigning a linear order $\trleft_e$ on these edges. This produces a multigraph $\Gscr=(\Vscr,\Escr)$ and determines the graph homomorphism $h:\Gscr\to G$ mapping the vertices and edges of $\Gscr$ to their original ones in $G$. The map $h$ generates a correspondence between values on edges of $G$ and tuples of edges in $\Gscr$; namely, for an edge $e\in E$ and a number $\alpha\in\Zset_+$ not exceeding $b(e)$, we denote by $h^{-1}(e,\alpha)$ the set of $\alpha$ edges in $h^{-1}(e)$ smaller by $\trleft_e$. In addition, for each vertex $\nu\in\Vscr$, we assign a certain linear order $>_\nu$ on the edges of $\Gscr$ incident to $\nu$.

The obtained bipartite graph $\Gscr$ with the vertex parts $h^{-1}(W)$ and $h^{-1}(F)$ together with the orders $>_\nu$ for $\nu\in\Vscr$ forms an instance of a (multigraph) version of the classical stable marriage problem by Gale and Shapley.  The multigraph $\Gscr$ has at most $\sum(\theta(v)\colon v\in V)+|V|$ vertices and $\sum(b(e)\colon e\in E)$ edges, and we are able to arrange $(\Gscr,\trleft,>)$ in such a way that the lattice $(\Sscr,\prec)$ is embedded as a sublattice into the lattice of stable matchings for $(\Gscr,>)$. More precisely, for each $x\in\Sscr$, the desired stable matching is defined to be the union of sets $h^{-1}(e,x(e))$ over all $e\in E$. See Theorem~\ref{tm:split}. Note that this embedding need not be an isomorphism, i.e. stable matchings for $(\Gscr,>)$ not corresponding to stable g-allocations for $(G,b,C)$ are possible. When all capacities are ones, a similar result (in terms of ``detachments'') was obtained in~\cite[Th.~5.3]{flein}. 

In the second, more compact, construction, we apply an \emph{aggregated} splitting techniques which transforms SGAP with $(G,b,C)$ into a stable allocation problem on a multigraph $\hat\Gscr=(\hat\Vscr,\hat\Escr)$ with edge capacities $\hat b$, vertex quotas $\hat q$, and linear orders $\hat>_\nu$ on the incident edges of vertices $\nu\in\hat\Vscr$. The aim is to provide that each stable g-allocation for $(G,b,C)$ turns into a stable (usual) allocation for $(\hat\Gscr,\hat b,\hat q,\hat >)$. In general, the size of $\hat\Gscr$ becomes smaller compared with that of $\Gscr$ in the previous construction; we show that $|\hat\Vscr|$ is at most $2|E|$ and $|\hat \Escr|$ is estimated as $O(N^2|E|)$, where $N$ is the length of a full route for $(G,b,C)$. See Theorem~\ref{tm:aggregate_split} and the last paragraph in Sect.~\SEC{aggregate}.

This paper is organized as follows. Section~\SEC{back} contains basic definitions and backgrounds. Section~\SEC{split} explains what we mean under splitting in the first construction and state the main result on it, in Theorem~\ref{tm:split}. A review of results from~\cite{karz} needed for further considerations (involving rotations, routes, poset representation) is given in Sect.~\SEC{review}. Using this, we prove Theorem~\ref{tm:split} in Sect.~\SEC{proof}. The final Sect.~\SEC{aggregate} is devoted to the construction of aggregated splitting and proves the main result on it (Theorem~\ref{tm:aggregate_split}).


\section{Backgrounds}  \label{sec:back}



In this section we review some basic definitions and facts from~\cite{AG} needed to us.

We consider a bipartite graph $G=(V,E)$ in which the vertex set $V$ is
partitioned into two parts (independent sets, color classes) $W$ and $F$, conditionally called the sets of \emph{workers} and \emph{firms}, respectively. The edges $e\in E$ have nonnegative integer \emph{capacities} $b(e)\in\Zset_{+}$.
Unless explicitly said otherwise, we assume that $G$ has no multiple edges. The edge connecting vertices $w\in W$ and $f\in F$ may be denoted as $wf$.

For a vertex $v\in V$, the set of its incident edges is denoted by $E_v$. We write $\Bscr=\Bscr^{(b)}$ for the nonnegative \emph{integer} box $\{x\in \Zset_+^E\colon x\le b\}$. For $v\in V$, the restriction of $\Bscr$ to the set $E_v$ is denoted by $\Bscr_v$.
\medskip

$\bullet$ (\textbf{choice functions}) Each vertex (``agent'') $v\in V$ can prefer one vector in $\Bscr_v$ to another. The preferences depend on a \emph{choice function} (CF) related to $v$. This is a map $C=C_v$ of $\Bscr_v$ into itself satisfying $C(z)\le z$ for all $z\in \Bscr_v$. Also $C$ obeys two additional axioms. One of them considers $z,z'\in\Bscr_v$ and requires that:
  \begin{itemize}
  %
\item[(A1)] if $z\ge z'$, then $C(z)\wedge z'\le C(z')$ (\emph{substitutability}, or
\emph{persistence}).
  \end{itemize}

\noindent(Hereinafter, for $a,b\in\Rset^S$, the functions $a\wedge b$ and $a \vee b$  take the values $\min\{a(e),b(e)\}$ and $\max\{a(e),b(e)\}$, $e\in S$, respectively.)

The other axiom imposed on $C=C_v$ for each $v\in V$ is called the \emph{size monotonicity} condition; it requires that
 \begin{itemize}
 \item[(A2)] if $z\ge z'$, then $|C(z)|\ge |C(z')|$.
 \end{itemize}
 
\noindent(Hereinafter, for a numerical function $a$ on a finite set $S$, we write $a(S)$ for $\sum(a(e) : e\in S)$, and $|a|$ for $\sum(|a(e)|\colon e\in S)$. In particular,
$a(S)=|a|$ if $a$ is nonnegative.)

A special case of~(A2) is the condition of \emph{quota filling}; it is applied when a vertex $v\in V$ is endowed with a \emph{quota} $q(v)\in \Zset_+$  and reads as:
 \begin{itemize}
 \item[(A3)]  any $z\in\Bscr_v$ satisfies $|C(z)|=\min\{|z|,q(v)\}$.
 \end{itemize}

\noindent\textbf{Example 1.} 
A popular instance of choice functions $C_v$ obeying axioms (A1) and (A3) is generated by a \emph{linear order} $>_v$ on the set $E_v$. Here
for $e,e'\in E_v$ with $e>_v e'$, the ``agent'' $v$ is said to prefer the edge (``contract'') $e$ to $e'$. Then $C_v$ related to  the order $>_v$ and a quota $q(v)\in\Zset_+$ is defined by the following rule: for $z\in\Bscr_v$, if $|z|\le q(v)$, then $C_v(z):=z$; and if $|z|> q(v)$, then, renumbering the edges in $E_v$ as $e_1,\ldots,e_{|E_v|}$ so that $e_i>_v e_{i+1}$ for each $i$, take the maximal $j$ satisfying $r:=\sum(z(e_i)\colon i\le j)\le q(v)$ and define $C_v(z)$ to be $(z(e_1),\ldots,z(e_j),q(v)-r,0,\ldots,0)$.
 \medskip

$\bullet$ (\textbf{stability}) For a vertex $v\in V$ and a vector (function, assignment) $x$ on $E$, let $x_v$ denote the restriction of $x$ to  $E_v$. A vector $z\in \Bscr_v$ is called  \emph{acceptable} if $C_v(z)=z$; the set of such vectors is
denoted by $\Ascr_v$. This notion is extended to $\Bscr$; namely, we
say that $x\in\Bscr$ is (globally) acceptable if $x_v\in\Ascr_v$ for all $v\in V$. The set of acceptable vectors in $\Bscr$ is denoted by $\Ascr$.

For any $v\in V$, the CF $C_v$ establishes preference relations on acceptable vectors on $E_v$; namely, $z\in\Ascr_v$ is said to be \emph{preferred} to $z'\in\Ascr_v-\{z\}$ if
   \begin{equation} \label{eq:zzp}
   C_v(z\vee z')=z,
   \end{equation}
denoted as $z'\prec_v z$. The relation $\prec_v$ is shown to be transitive.

The preferences between acceptable functions on the sets $E_v$, $v\in
V$, are extended, in a natural way, to the whole $E$. Namely, for the ``firm'' part $F$ in $V$ and distinct $x,y\in \Ascr$, we write $x\prec_F y$ if $x_f\preceq_f y_f$ holds for all $f\in F$. The preferences in $\Ascr$ relative to the
``worker'' part  $W$ are defined in a similar way and denoted via $\prec_{\,W}$.

Given $G$, $b$  and $C_v$ ($v\in V$) as above, we will refer to functions $x\in\Ascr$ on $E$ as (acceptable) \emph{generalized allocation}, or \emph{g-allocations} for short. 
 \medskip

\textbf{Definition 1.} Let $v\in V$ and $z\in\Ascr_v$. Using terminology from~\cite{karz}, we say that an edge $e\in E_v$ is \emph{interesting} for $v$ under $z$ if there exists $z'\in\Bscr_v$ such that
  \begin{equation} \label{eq:inter_e}
  z'(e)>z(e),\quad \mbox{$z'(e')=z(e')$ for all $e'\ne e$,}\quad \mbox{and $C_v(z')(e)>z(e)$}
  \end{equation}
(the term ``interesting'' is justified by an expectation that an increase at $e$ can give rise to a better assignment for $v$). Extending this to g-allocations on $E$, we say that an edge $e=wf\in E$ is \emph{interesting} for a vertex (``agent'') $v\in\{w,f\}$ under a g-allocation $x\in\Ascr$ if so is for $v$ and $x_v$. If  $e=wf\in E$
is interesting under $x$ for both vertices $w$ and $f$, then the edge $e$ is
called \emph{blocking} for $x$. A g-allocation $x\in \Ascr$ is called
\emph{stable} if no edge in $E$ blocks $x$. The set of stable g-allocations 
is denoted by $\Sscr=\Sscr_{G,b,C}$.
\medskip

$\bullet$ ~For $v\in V$, the set $\Ascr_v$ endowed with the preference relation
$\succ_v$ turns into a \emph{lattice}. In this lattice, for $z,z'\in\Ascr_v$, the
least upper bound (join) $z\curlyvee z'$ is expressed as $C_v(z\vee z')$, and
the greatest lower bound (meet) $z\curlywedge z'$ is expressed as $C_v(\bar z\wedge \bar z')$, where $\bar y$ is the \emph{closure} of $y\in\Ascr_v$, defined as $\sup(y'\in\Bscr_v\colon C_v(y')=y)$.
 \medskip

$\bullet$ ~In a general setting, Alkan--Gale's model in~\cite{AG} deals with a bipartite graph $G=(V,E)$ and domains $\Bscr_v$ ($v\in V$) that are closed subsets of the real boxes $\{z\in \Rset_+^{E_v}\colon z(e)\le b(e)\; \forall e\in E_v\}$ for some $b\in \Rset_+^E$. Also one assumes that each choice function $C_v$ ($v\in V$) acting on $\Bscr_v$ is continuous and obeys the axioms as above. So our model (or problem) of stable g-allocations is an ``integer version'' of that; we denote it as \emph{SGAP}. General results in~\cite{AG} imply that the set $\Sscr$ of stable g-allocations in SGAP possesses the following nice properties that will be important to us in what follows:
 \begin{numitem1} \label{eq:AG}
 \begin{itemize}
\item[(a)] $\Sscr$ is nonempty, and  $(\Sscr,\prec_F)$ is a
\emph{distributive} lattice (cf.~\cite[Ths.~1,8]{AG});
\item[(b)]
(\emph{polarity}): $\prec_F$ is opposite to 
$\prec_W$ on $\Sscr$; namely: for $x,y\in\Sscr$, if  $x_f\preceq_f y_f$ for all $f\in F$, then $y_w\preceq_w x_w$ for all $w\in W$, and vice versa (cf.~\cite[Cor.~2]{AG});
\item[(c)]
(\emph{unisizeness}): for each vertex $v\in V$, the size $|x_v|$ is
the same for all stable g-allocations $x\in \Sscr$ (cf.~\cite[Th.~6]{AG}).
 \end{itemize}
   \end{numitem1}

We denote the minimal and maximal elements in the lattice $(\Sscr,\prec_F)$ by
$\xmin$ and $\xmax$, respectively; then the former is the best and the latter
is the worst for the part $W$, in view of the polarity~(2.3)(b).

In what follows, 
we may write $\onebf^e_v$ for the unit base vector of $e$ in $\Rset^{E_v}$ (taking value 1 on $e$, and 0 otherwise). The term $v$ may be omitted here if it is clear from the context.

 
 \section{Splitting} \label{sec:split}
 
Let $G=(V=W\sqcup F,E),\,b,\,C$ be as above (where $\sqcup$ means the disjoint union).
\medskip

 \noindent\textbf{Definition 2.}
Let $\theta: V\to \Zset_+$  satisfy $\theta(v)\le b(E_v)$ for each $v\in V$. By a $\theta$-\emph{splitting} of $G$ we mean a bipartite multigraph $\Gscr=(\Vscr=\Wscr\sqcup\Fscr,\Escr)$ along with a map (\emph{graph homomorphism}) $h: \Gscr\to G$ such that:
  \begin{numitem1} \label{eq:split}
  \begin{itemize}
\item[(a)] $h(\Wscr)=W$, $h(\Fscr)=F$, and $h(\Escr)= E$;
\item[(b)] for each $v\in V$, the set $\Vscr_v:=h^{-1}(v)$ has size $\max\{\theta (v),1\}$;
\item[(c)] for each edge $e=uv\in E$, the set $\Escr(e):=h^{-1}(e)$ consists of $b(e)$ edges, each connecting a vertex in $\Vscr_u$ with a vertex in $\Vscr_v$, and these edges are endowed with a linear order $\trleft_e$ (where for $\eps,\eps'\in\Escr(e)$ with $\eps\trleft_e \eps'$, we say that $\eps$ is \emph{smaller}, or \emph{appears earlier} than $\eps'$);
\item[(d)] for each vertex $\nu\in\Vscr$, the set $\Escr_\nu$ of edges of $\Gscr$ incident to $\nu$ is endowed with a linear order $>_\nu$ establishing preferences on $\Escr_\nu$ (where for $\eps,\eps'\in\Escr_\nu$ with $\eps>_\nu \eps'$, we say that $\nu$ \emph{prefers} $\eps$  to $\eps'$).
 \end{itemize}
  \end{numitem1}
  
  \noindent\textbf{Definition 3.}
Let $x\in \Zset^E_+$ satisfy $x(e)\le b(e)$ for all $e\in E$. We associate to $x$ the set $\Mscr_x=\Mscr_{x,h}\subseteq\Escr$ such that for each $e\in E$, $\Mscr_x$ consists of $x(e)$ smaller edges of $\Escr(e)$ (w.r.t. the order $\trleft_e$).

 %
 \begin{theorem} \label{tm:split}
For each $v\in V$, define $\theta(v):=|x_v|$, where $x$ is a stable g-allocation for $(G,b,C)$ as above ($\theta$ is well-defined by~\refeq{AG}(c)). Then there exists a $\theta$-splitting $(\Gscr=(\Vscr,\Escr),h,\trleft,>)$ that possesses the following property: for each stable $x\in\Sscr_{G,b,C}$, the set $\Mscr_x$ forms a matching in $\Gscr$ (i.e. no two edges of $\Mscr_x$ have a common endvertex) and, moreover, $\Mscr_x$ is stable for $\Gscr$ under the preferences $>_\nu$, $\nu\in \Vscr$. Furthermore, under the correspondence $x\mapsto \Mscr_x$, the lattice of stable g-allocations for $(G,b,C)$ is mapped into a sublattice of the lattice of stable matchings for $(\Gscr,>)$.
 \end{theorem}
 
This theorem will be proved in Sect.~\SEC{proof} relying on some constructions and results from~\cite{karz} reviewed in the next section.


\section{A review on the stable g-allocation model} \label{sec:review}


\subsection{Rotations.} \label{ssec:rotat}
~By an (abstract) \emph{rotation} we mean a cycle $R=(v_0,e_1,v_1,\ldots, e_k,v_k=v_0)$ in $G$ in which the edges $e_1,\ldots,e_k$ are different, i.e. $R$ is edge-simple (but not necessarily simple, as some vertices may be passed several times). We write $V_R$ and $E_R$ for the sets of (different) vertices and edges in $R$, respectively. An edge $e_i$ is called \emph{positive} (\emph{negative}) in $R$ if the vertex $v_{i-1}$ belongs to the part $W$ (resp. $F$). The set of positive (negative) edges of $R$ is denoted by $R^+$ (resp. $R^-$), and we associate to $R$ the incidence $0,\pm 1$ vector $\chi^R\in \Zset^E$ taking value 1 on the posititive edges, $-1$ on the negative edges, and 0 on the other edges of $G$. 

As is shown in~\cite{karz}, for each stable g-allocation $x\in\Sscr$, there exists (and can be efficiently constructed) a set $\Rscr(x)$ of rotations possessing the following nice properties:
  \begin{numitem1} \label{eq:rot_prop} 
    \begin{itemize}
\item[(i)]
for each $R\in\Rscr(x)$, the vector $x':=x+\chi^R$ is stable and it immediately succeeds $x$ in the lattice $(\Sscr,\prec_F)$, in the sense that $x\prec_F x'$ and there is no $y\in\Sscr$ such that $x\prec_F y\prec_F x'$ (see~\cite[Prop.~3.4]{karz});
 \item[(ii)]
conversely, for each stable vector $x'$ that immediately succeeds $x\in\Sscr$, there exists a rotation $R\in\Rscr(x)$ such that $x'=x+\chi^R$ (see~\cite[Prop.~3.5]{karz});
 \item[(iii)] the edge sets of rotations in $\Rscr(x)$ are pairwise disjoint (see~\cite[Sec.~3.1]{karz}).
 \end{itemize}
 \end{numitem1}

A rotation $R\in\Rscr(x)$ is called  \emph{applicable} to $x\in\Sscr$; we also say that the stable vector $x':=x+\chi^R$ is obtained from $x$ by applying the rotation $R$ (with weight 1) and that $R$ is \emph{increasing}. (This term concerns the order $\prec_F$, in view of $x\prec_F x'$.) Applying the \emph{reversed} rotation $\bar R=(v_k,e_k,v_{k-1},\ldots,e_1,v_0)$ to $x'$ returns $x$, namely, $x=x'+\chi^{\bar R}$, and we say that $\bar R$ is \emph{decreasing} (w.r.t. $\prec_F$). There is a unique stable vector admitting no increasing (resp. decreasing) rotation, namely, $\xmax$ (resp. $\xmin$).
 \medskip
 
\noindent\textbf{Remark 1.} 
An increasing rotation $R$ applicable to $x\in\Sscr$ is defined up to shifting cyclically. A priori we cannot exclude the existence of another rotation $R'$ (applicable to some $y\ne x$) having the same positive and negative parts: $R'^+=R^+$ and $R'^-=R^-$. A useful property (see~\cite[Sec.~3.1]{karz}) is that the cycle (increasing rotation) $R$ is determined by $x$ and one edge $e$ in it. More precisely,
  \begin{numitem1} \label{eq:e-ep}
if $e$ is a positive edge of an increasing rotation $R$ connecting vertices $w\in W$ and $f\in F$, then $e$ is interesting for $f$ under $x$ and satisfies $C_f(x_f+\onebf^e_f)=x_f+\onebf^e_f-\onebf^{e'}_f$, where $e'$ is the next (negative) edge in the cycle $R$;
  \end{numitem1}
When $e$ is negative, the (positive) edge $e'$ next to $e$ in $R$ is determined by $e$ as well, but in a somewhat different way than in~\refeq{e-ep}; see~\cite[Exp.~(3.3)]{karz}. 

Note also that, using~\refeq{e-ep}, one can show the following (cf.~\cite[Lemma~3.6]{karz}):
  \begin{numitem1} \label{eq:CfR+}
for an increasing rotation $R$ applicable to a stable $x$ and passing a vertex $f\in F$, there holds $C_f(x_f+\chi_f^{R^+})=x_f+\chi_f^{R^+}-\chi_f^{R^-}$, where $\chi_f^{R^+}$ (resp. $\chi_f^{R^-}$) is the sum of vectors $\onebf_f^e$ over $e\in R^+\cap E_f$ (resp. $e\in R^-\cap E_f$) 
\end{numitem1}

Next, for an increasing rotation $R$ applicable to a stable $x\in\Sscr$, we may try to apply $R$ with a larger weight. More precisely, we say that a weight $\lambda\in\Zset_{>0}$ is \emph{feasible} for $(x,R)$ if $R$ can be applied with weight 1, step by step, $\lambda$ times, which means that the sequence $x=x_0,x_1,\ldots,x_\lambda$ defined by $x_i:=x_{i-1}+\chi^R$, $i=1,\ldots,\lambda$, consists of stable vectors. (In this case $x_{i-1}\prec_F x_i$ is valid for each $i$.) The maximal feasible weight for $(x,R)$ is denoted by $\tau_R(x)$.

Consider the set $\Rscr(x)$ of (increasing) rotations applicable to $x\in\Sscr$. An important fact is that the rotations in $\Rscr(x)$ commute. Moreover, the following property is valid (see~\cite[Cor.~4.1]{karz}).
  \begin{numitem1} \label{eq:commute}
Let $\Rscr'\subseteq\Rscr(x)$ and let $\lambda:\Rscr'\to \Zset_+$ be such
that $\lambda(R)\le \tau_R(x)$ for each $R\in\Rscr'$. Then the vector
$x':=x+\sum(\lambda(R) \chi^R \colon R\in\Rscr')$ is stable, each $R\in \Rscr'$ with $\lambda(R)<\tau_R(x)$ is an increasing rotation applicable to $x'$ having the maximal feasible weight $\tau_R(x')=\tau_R(x)-\lambda(R)$, and each $R'\in \Rscr(x)-\Rscr'$ is applicable to $x'$ keeping the maximal feasible weight: $\tau_{R'}(x')=\tau_{R'}(x)$. Roughly speaking, the rotations in $\Rscr(x)$ can be applied in any order.  
  \end{numitem1}


\subsection{Routes.} \label{ssec:routes}
~Let $\Tscr$ be a sequence $x_0,x_1,\ldots,x_N$ of stable vectors such that each $x_i$ is obtained from $x_{i-1}$ by applying an increasing rotation $R_i$ with a feasible weight $\lambda_i\in\Zset_{>0}$, i.e. $R_i\in\Rscr(x_{i-1})$, $\lambda_i\le \tau_{R_i}(x_{i-1})$ and $x_i=x_{i-1}+\lambda_i\chi^{R_i}$. Then $x_0\prec_F x_1\prec_F\cdots \prec_F x_N$. We call $\Tscr$ a \emph{route} from $x_0$ to $x_N$. From~\refeq{rot_prop} it follows that there exists a route from $\xmin$ to $\xmax$. We liberally say that a rotation $R_i$ with weight $\lambda_i$ \emph{is used} in $\Tscr$.

Note that one and the same rotation may be used in a route many times. Using terminology from~\cite{karz}, a route $\Tscr$ as above  is called: \emph{non-excessive} if $i<j$ and $R_i=R_j$ imply $\lambda_i=\tau_{R_i}(x_{i-1})$;  and \emph{principal} if $\lambda_i=\tau_{R_i}(x_{i-1})$ for all $i=1,\ldots,N$. A principal route from $\xmin$ to $\xmax$ is called \emph{full}. One shows that for any $x,y\in \Sscr$ with $x\prec_F y$, there exists a non-excessive route from $x$ to $y$.

For a non-excessive route $\Tscr$, let $\Pi(\Tscr)$ denote the \emph{family} of pairs $(R_i,\lambda_i)$ (\emph{weighted rotations}) used in $\Tscr$ (where each pair $(R,\lambda)$ occurs in $\Pi$ as many times as it is used in $\Tscr$). The following property is of importance (see~\cite[Prop.~4.2]{karz}):
  \begin{numitem1} \label{eq:invar_rot}
Let $x,y\in \Sscr$ and $x\prec_F y$. Then for all non-excessive routes $\Tscr$ going from $x$ to $y$, the family $\Pi(\Tscr)$ is the same. A similar property is valid relative to principal routes as well.
 \end{numitem1}
\noindent(Initially an invariance property of this sort (in case $(x,y)=(\xmin,\xmax)$) was revealed by Irving and Leather~\cite{IL} for usual rotations in the classical stable marriage problem and subsequently it was shown for more general models of stability.)

Let us call a stable vector $x\in \Sscr$ \emph{principal} if there is a principal route from $\xmin$ to $x$. We denote the set of principle vectors by $\Sscrpr$, and the restriction of $\prec_F$ to $\Sscrpr$ by $\precpr$. (Then $(\Sscrpr,\precpr)$ forms a distributive sublattice of $(\Sscr,\prec_F)$.)


\subsection{Poset of rotations and closed functions.} \label{ssec:poset_rot}
~By~\refeq{invar_rot}, the family $\Pi(\Tscr)$ of pairs (weighted rotations) $(R,\tau)$ that are used (respecting possible replications) in a full route $\Tscr$ does not depend on the full route; we abbreviate it as $\Pi$. We write $\Rscr$ for the set of \emph{different} rotations used in a full route; equivalently, $\Rscr$ is the set of  different cycles in $G$ forming increasing rotations applicable to stable vectors. For $R\in\Rscr$, the subfamily of pairs in $\Pi$ involving this $R$ is denoted by $\Pi_R$. We also denote by $\hat\Rscr$ the family of all rotations, with possible replications, occurring in $\Pi$ (then for $R\in\Rscr$, there are as many occurrences of $R$ in $\hat\Rscr$ as the cardinality of $\Pi_R$).

A crucial fact is that one can arrange a poset on $\Pi$ that generates a representation for the principal lattice $(\Sscrpr,\precpr)$ and, further, for the whole lattice $(\Sscr,\prec_F)$. (The latter extends a classical result by Irving et al.~\cite{ILG} on a poset representation for stable marriages.) More precisely (see~\cite[Th.~5.9, Cor.~5.10, Exp.~(5.6)]{karz}),
  \begin{numitem1} \label{eq:princ_biject}
there exist a partial order $\lessdot$ on $\Pi$ and a map $\phipr$ from the principal vectors to subsets of $\Pi$ such that:
 \begin{itemize}
 \item[(a)] for each $x\in\Sscrpr$, $\phipr(x)$ is the family of weighted rotations $(R,\tau)$ used in a principal route from $\xmin$ to $x$ (i.e.  $x=\xmin+\sum(\tau\chi^R \colon (R,\tau)\in \phipr(x))$);
 \item[(b)] $\phipr$ establishes an isomorphism between the principal lattice $(\Sscrpr,\precpr)$ and the lattice of closed subfamilies in $(\Pi,\lessdot)$, where $\Cscr\subseteq \Pi$ is called \emph{closed} if $(R,\tau),(R',\tau')\in \Pi$, $(R,\tau)\lessdot (R',\tau')$  and $(R',\tau')\in \Cscr$ imply $(R,\tau)\in\Cscr$, and where the closed subfamilies are regarded as partially ordered by inclusion;
\item[(c)] for each rotation $R\in\Rscr$, the restriction of $\lessdot$ to $\Pi_R$ is a linear order.
 \end{itemize}
 \end{numitem1}
 
\noindent In particular, under the bijection $\phipr$ between the principal stable vectors and the closed families in $(\Pi,\lessdot)$, $\xmin$ corresponds to the empty set, and $\xmax$ to the whole $\Pi$. 

Equivalently, the partial order $\lessdot$ can be defined as follows (in a spirit of~\cite{ILG}):
 \begin{numitem1} \label{eq:lessdot}
pairs $(R,\tau),(R',\tau')\in\Pi$ satisfy $(R,\tau)\lessdot (R',\tau')$ if and only if for \emph{any} full route $\Tscr$, the pair $(R,\tau)$ is used in $\Tscr$ \emph{earlier} than $(R',\tau')$ (see~\cite[Secs.~5.1,\,5.2]{karz}).
  \end{numitem1}

Next we extend $\phipr$ to a map from all stable vectors. Let us call a function $\lambda:\Pi\to\Zset_+$ \emph{closed} if it does not exceed $\tau$ (i.e. $\lambda(R,\tau)\le \tau$ for all $(R,\tau)\in\Pi$) and the relations $(R,\tau)\lessdot (R',\tau')$ and $\lambda(R',\tau')>0$ imply $\lambda(R,\tau)=\tau$. 

One can see that for any closed function $\lambda$, the family $\{(R,\tau)\in\Pi\colon \lambda(R,\tau)>0\}$ is closed, and ``conversely'': taking a closed family $\Cscr\subseteq\Pi$, one can form a closed function $\lambda$ by defining $\lambda(R,\tau)$ to be an arbitrary integer between 0 and $\tau$ for each \emph{maximal} pair $(R,\tau)$ in $\Cscr$, the value $\tau$ for the other pairs $(R,\tau)$ in $\Cscr$, and 0 for the rest. As to stable vectors, for any $x\in\Sscr$, taking a non-excessive route $\Tscr$ from $\xmin$ to $x$ and relying on~\refeq{commute}, one can transform $\Tscr$ into a principal route in a natural way, by assigning the maximal feasible weight of the current rotation at each step, thus associating to $x$ the resulting principal vector $x'$. Based on these observations, the following facts can be obtained (see~\cite[Th.~5.11]{karz}):
  \begin{numitem1} \label{eq:gen_biject}
there exists a map $\phi$ from the stable vectors to functions on $\Pi$ such that:
  \begin{itemize}
\item[(a)] for each $x\in\Sscr$, $\phi(x)$ is generated by a non-excessive route $\Tscr$ from $\xmin$ to $x$; namely, for each pair $(R,\tau)\in\Pi$, $\phi(x)$ takes value $\lambda$ if $R$ (as an element of $\hat\Rscr$) is used with weight $\lambda$ in $\Tscr$, and 0 otherwise;
 \item[(b)] $\phi$ establishes an isomorphism between the lattice $(\Sscr,\prec_F)$ and the lattice of closed functions for $(\Pi,\lessdot)$ (using the natural comparison $\le$ on functions);
 \item[(c)] the restriction of $\phi$ to $\Sscrpr$ corresponds to $\phipr$. 
 \end{itemize}
  \end{numitem1}  
  
Basic tasks used in our constructions concern finding rotations and their maximal weights. As is shown in~\cite[Sec.~6]{karz},
  \begin{numitem1} \label{eq:2tasks}
given a stable vector $x\in\Sscr$, 
  \begin{itemize}
  \item[(a)] the set $\Rscr(x)$ of rotations applicable to $x$ can be constructed in strongly polynomial time; in particular, it takes $O(|E|^2)$ oracle calls;
  \item[(b)] for each $R\in\Rscr(x)$, the maximal feasible weight $\tau_R(x)$ can be found in pseudo polynomial time, viewed as $O(\bmax|E|^2)$.
    \end{itemize}
    \end{numitem1}
    
An important parameter in estimates in~\cite{karz} is the length $N$ of a full route $\Tscr=(x_0,x_1,\ldots,x_N)$, which is equal to the number $|\Pi|$ of elements in the poset $(\Pi,\lessdot)$, or the size $|\hat\Rscr|$. One shows that
  \begin{numitem1} \label{eq:N}
  \begin{itemize}
 \item[(a)] the poset $(\Pi,\lessdot)$ can be constructed in pseudo polynomial time; in particular, the number of oracle calls is $O(|E|^2(\bmax N+N^2))$ (see~\cite[Th.~6.1]{karz});
\item[(b)] $N\le \bmax |E|/2$. 
   \end{itemize} 
   \end{numitem1}
Here~(b) is a consequence from the following useful property (shown in the proof of Lemma~6.2 in~\cite{karz}) :
   \begin{numitem1} \label{eq:one_peak}
for $e\in E$ and a full route $\Tscr=(x_0,x_1,\ldots,x_N)$, the values $x_0(e),x_1(e),\ldots,x_N(e)$ behave in a ``one-peak'' manner; namely, there is $i$ such that the values from $x_0(e)$ to $x_i(e)$ are weakly increasing (admitting equalities), and the values from $x_i(e)$ to $x_N(e)$ are weakly decreasing.
  \end{numitem1}

 
 \section{Proof of Theorem~\ref{tm:split}} \label{sec:proof}
 
Let $\theta$ be as in the hypotheses of this theorem. In our proof we will use a weakened version of the notion of $\theta$-splitting, called a \emph{partial $\theta$-splitting}. The difference from the one in Sect.~\SEC{split} is as follows. First, in Definition~2, we replace $\Escr$ by a (possibly smaller) set $\Ascr=\cup(\Ascr(e)\colon e\in E)$ assuming that $h(\Ascr)\subseteq E$; here (cf.~\refeq{split}(c)) for each $e=uv\in E$, the set $\Ascr(e)=h^{-1}(e)$ consists of \emph{at most} $b(e)$ edges, each connecting a vertex in $h^{-1}(u)$ with a vertex in $h^{-1}(v)$, and accordingly, the elements of $\Ascr(e)$ are linearly ordered by $\trleft_e$. Second, in ~\refeq{split}(d), for a vertex $\nu\in\Vscr$, a linear order $>_v$ is assigned on the set $\Ascr_\nu$ of edges in $\Ascr$ incident to $\nu$ (rather than $\Escr_\nu$). Third, in Definition~3, one should consider $x\in\Zset_+^E$ such that $x(e)\le|\Ascr(e)|$ for all $e\in E$, and we associate to $x$ the corresponding subset $\Mscr_x$ of $\Ascr$.

To prove the theorem, we take an arbitrary full route $\Tscr=(\xmin=x^0,x^1,\ldots,x^N=\xmax)$ where each $x^i$ ($0<i\le n$) is obtained from $x^{i-1}$ by applying a rotation $R_i$ with the maximal feasible weight $\tau_{R_i}$. For $i=0,\ldots,N$, define $y^i\in \Zset_+^E$ by
  \begin{equation} \label{eq:yi}
  y^i(e):=\max\{x^0(e),\ldots,x^i(e)\}\quad\mbox{for each}\quad e\in E.
  \end{equation}
Then $y^0\le y^1\le \cdots\le y^N$ (we shall see later that all inequalities here are strong). 

Using $N+1$ iterations, we construct partial $\theta$-splittings $\Gscr^i=(\Vscr,\Ascr^i)$, $i=0,\ldots,N$, along with graph homomorphisms $h^i:\Gscr^i\to G$ (where $h^0(\Vscr)=\cdots=h^N(\Vscr)$) and linear orders $\trleft^i_e$ ($e\in E$) and $>^i_\nu$ ($\nu\in\Vscr$) satisfying 
 \begin{equation} \label{eq:y-A}
  |\Ascr^i(e)|=y^i(e),\;\;i=0,\ldots,N,\; e\in E,\quad\mbox{and}\quad \Ascr^0\subseteq \cdots\subseteq \Ascr^N.
  \end{equation}
  
For $i$ and $e$ as above, we define $\Xscr^i(e)$ to be the set (sublist) of $x^i(e)$ \emph{smaller} elements of $\Ascr^i(e)$ (w.r.t. the order $\trleft^i_e$).
\medskip
  
Initially, at \emph{iteration 0}, for each $e=uv\in E$, we assign as $\Ascr^0(e)$ (equal to $\Xscr^0(e)$) an arbitrary collection (list) of $x^0(e)$ pairwise disjoint edges connecting $h^{-1}(u)$ and $h^{-1}(v)$ and assign a linear order $\trleft^0_e$ on them, in such a way that the whole set $\Ascr^0=\cup(\Ascr^0(e))$ forms a matching, denoted as $\Mscr^0$ (that is just $\Mscr_{x^0}$), which is, obliously, possible. 

Then for each $v\in V$, the set $\Vscr_v=h^{-1}(v)$ consists of $\theta(v)$ elements, and for each $\nu\in\Vscr_v$, the set $\Ascr^0_\nu$ consists of a single element (so the order $>^0_\nu$ on $\Ascr^0_\nu$ is trivial).
\medskip

Next we describe \emph{$i$-th iteration} for $i=1,\ldots,N$. It transforms $\Ascr^{i-1}$ into $\Ascr^i$ by using the rotation $R_i$ with the weight $\tau:=\tau_{R_i}$. Let $a_1,c_1,a_2,c_2,\ldots,a_k,c_k$ be the sequence of edges of $R_i$ (up to shifting cyclically), where for $j=1,\ldots,k$, the edge $a_j$ ($c_j$) belongs to $R^+_i$ (resp. $R^-_i$) and connects vertices $w_j\in W$ and $f_j\in F$ (resp. $f_j$ and $w_{j+1}$), letting $w_{k+1}:=w_1$. We first explain how to transform each set $\Xscr^{i-1}(c_j)$ into $\Xscr^i(c_j)$, and then how to transform each $\Xscr^{i-1}(a_j)$ into $\Xscr^i(a_j)$.

For $j=1,\ldots,k$, the negativity of $c_j$ in $R_i$ implies that $|\Xscr^{i-1}(c_j)|=x^{i-1}(c_j)\ge \tau$. Therefore, we can extract from $\Xscr^{i-1}(c_j)$ the set (sublist) $\Xi_j$ of $
\tau$ \emph{larger} edges (w.r.t. the order $\trleft^{i-1}_{c_j}$), say, $\xi^1_j=w^1_{j+1}f^1_j,\ldots,\xi^\tau_j=w^\tau_{j+1}f^\tau_j$. Here $w^1_{j+1},\ldots,w^\tau_{j+1}$ are different vertices in $\Wscr$ (copies of $w_{j+1}$), and  $f^1_j\ldots,f^\tau_j$ are different vertices in $\Fscr$ (copies of $f_j$). 
Based on the relations
  \begin{equation} \label{eq:Xcj}
  |\Xscr^{i-1}(c_j)-\Xi_j|=x^{i-1}(c_j)-\tau=x^i(c_j),
  \end{equation}
we assign the set (list) $\Xscr^i(c_j)$ by removing $\Xi_j$ from $\Xscr^{i-1}(c_j)$. At the same time, the $\Ascr$-set and order $\trleft$ are preserved for $c_j$; namely, $\Ascr^i(c_j):=\Ascr^{i-1}(c_j)$ and $\trleft^i_{c_j}=\trleft^{i-1}_{c_j}$.

In turn, each positive edge $a_j=w_{j}f_j$ of $R_i$ ($j=1,\ldots,k$) satisfies
  \begin{equation} \label{eq:bxaj}
  b(a_j)\ge x^{i-1}(a_j)+\tau=x^i(a_j).
  \end{equation}
Relying on~\refeq{Xcj}--\refeq{bxaj}, we form a set $\Psi_j$ of new pairwise disjoint edges $\psi^1_j,\ldots,\psi^\tau_j$ (copies of $a_j$) connecting the $\Wscr$-vertices of $\Xi_{j-1}$ to the $\Fscr$-vertices of $\Xi_{j}$, respecting the orderings.  Namely, for $p=1,\ldots,\tau$, edge $\xi^p_j$  should connect $w^p_j$ and $f^p_{j}$. Then we assign $\Xscr^i(a_j)$ by adding $\Psi_j$ to the current $\Xscr^{i-1}(a_j)$, and accordingly assign $\Ascr^i(a_j)$ by adding $\Psi_j$ to $\Ascr^{i-1}(a_j)$, forming the new order $\trleft^i_{a_j}$ by adding the elements of $\Psi_j$ (in their order) to the end of $\trleft^{i-1}_{a_j}$.

The iteration finishes with updating the linear orders $>_\nu$ on the obtained sets $\Ascr^i_\nu$ for $\nu\in\Vscr$. Here for each new edge of the form $\psi^p_j=w^p_j f^p_{j}$ (where, as above, $j=1,\ldots,k$ and $p=1,\ldots,\tau$), we use the following rule:
  \begin{numitem1} \label{eq:lin_order}
$\psi^p_j$ is made least preferred for $\eta=w^p_j$ (in $\Ascr^i_{\eta}=\Ascr^{i-1}_{\eta}\cup\{\psi^p_j\}$) and most preferred for $\nu=f^p_{j+1}$ (in $\Ascr^i_{\nu}=\Ascr^{i-1}_{\nu}\cup\{\psi_j^p\}$).
  \end{numitem1}
  
Next we examine the collections $\Xscr^0,\ldots,\Xscr^N$ and $\Ascr^0,\ldots,\Ascr^N$ constructed in the above procedure. Considering $i$-th iteration for $i\ge 1$ and the sequence $a_1,c_1,\ldots,a_k,c_k$ of alternating positive and negative edges in the rotation $R_i$ and using notation as above, we define $\Xi^i$ and $\Psi^i$ to be the unions of sets $\Xi_j$ and $\Psi_j$ over $j=1,\ldots,k$, respectively. By the above construction, we have
   \begin{equation} \label{eq:XiAi}
   \Xscr^i=(\Xscr^{i-1}-\Xi^i)\sqcup \Psi^i\quad \mbox{and}\quad 
   \Ascr^i=\Ascr^{i-1}\sqcup \Psi^i=\Xscr^0\cup \cdots\cup\Xscr^i.
   \end{equation}
   
In particular, all inclusions in~\refeq{y-A} are strong. Also one can see from the construction that each vertex in $\Vscr$ is covered by edges in $\Xi^i$ as many times as by edges in $\Psi^i$. Using this and the fact that the initial set $\Xscr^0=\Ascr^0$ is a matching, we can conclude by induction on $i$ that
  \begin{numitem1} \label{eq:Xi-matching}
  all sets $\Xscr^0,\ldots, \Xscr^N$ are matchings covering the same set of vertices, namely, the set $\cup(h^{-1}(v)\colon v\in V,\; \theta(v)>0)$.
  \end{numitem1}
  
Let $\Ascr:=\Ascr^N$. By~\refeq{XiAi}, $\Xscr^i\subseteq \Ascr$ for $i=0,\ldots,N$. Also the construction guarantees that all edges in $\Ascr$ are different (a priori admitting multiple edges in it).

In order to obtain the required correspondence stated in Theorem~\ref{tm:split}, we first establish a similar interrelation concerning the stable g-allocations occurring in the above full route $\Tscr$ and the set $\Ascr$ (with linear orders on vertices and edges as above). We associate to each $x^i\in\Tscr$ the matching $\Xscr^i$, denoting the correspondence $x^i\mapsto\Xscr^i$ by $\gamma=\gamma_\Tscr$. Then $x^i(e)=|\Xscr^i(e)|$ for each $e\in E$ (cf.~\refeq{Xcj},\refeq{bxaj}).
 \begin{lemma} \label{lm:xi-Xi}
Let $x^i\in\Tscr$, Then:
  \begin{itemize}
\item[\rm(a)] for each $e\in E$, the set $\Xscr^i(e)$ consists of $x^i(e)$ first elements of $\Ascr(e)$ (w.r.t. $\trleft_e$);
 \item[\rm(b)] the matching $\Xscr^i$ is stable in $(\Vscr,\Ascr)$ w.r.t. the linear preferences $>_\nu$ on the sets $\Ascr_\nu$, $\nu\in\Vscr$.
   \end{itemize}
   \end{lemma}
   
\begin{proof}
To show (a) for $e\in E$, we use the ``one-peak'' property~\refeq{one_peak}, which implies the existence of $r\in\{0,\ldots,N\}$ such that $x^0(e)\le x^1(e)\le \cdots\le x^r(e)\ge x^{r+1}(e)\ge \cdots\ge x^N(e)$. By the construction, for $i=1,\ldots,r$, we have $\Xscr^{i-1}(e)\subseteq \Xscr^i(e)$ (for in case $x^{i-1}(e)<x^i(e)$, $e$ is positive in the rotation $R_i$, and $\Xscr^i(e)$ is obtained by adding $\tau_{R_i}$ new elements to the end of $\Xscr^{i-1}(e)$). And for $i=r+1,\ldots,N$, we have $\Xscr^{i-1}(e)\supseteq \Xscr^i(e)$ (for in case $x^{i-1}(e)>x^i(e)$, $e$ is negative in $R_i$, and $\Xscr^i(e)$ is obtained by deleting $\tau_{R_i}$ last elements from $\Xscr^{i-1}(e)$). This implies $\Xscr^r(e)=\Ascr(e)$, and~(a) follows.

To show~(b), consider a vertex $\nu$ in the ``firm'' part $\Fscr$ of $\Vscr$ and assume that $\Ascr_\nu$ is nonempty. Let $i(1)<i(2)<\cdots<i(p)$ be the sequence of indexes $i\in\{1,\ldots,N\}$ such that $\Xscr^{i-1}_\nu \ne \Xscr^{i}_\nu$. By the construction, for each $q=1,\ldots,p$, the difference between $\Xscr^{i(q)-1}_\nu$ and $\Xscr^{i(q)}_\nu$ must consist of two edges, of which one, $\xi(q)$ say, belongs to the set $\Xi^{i(q)}$ ($=\Xscr^{i(q)-1}-\Xscr^{i(q)}$), while the other, $\psi(q)$ say, belongs to $\Psi^{i(q)}$ ($=\Xscr^{i(q)}-\Xscr^{i(q)-1}$) (here the former (latter) is created from a negative (resp. positive) edge of the rotation $R_{i(q)}$; these edges are consecutive in $R_{i(q)}$ and pass the vertex $h(\nu)$ of $G$).

At $i(q)$-th iteration, the edge $\xi(q)$ in the current matching $\Xscr$ is replaced by $\psi(q)$. Moreover, by rule~\refeq{lin_order}, $\psi(q)$ becomes most preferred for $\nu$ in the current $\Ascr^{i(q)}_\nu=\Xscr^0_\nu \cup\cdots\cup \Xscr^{i(q)}_\nu$. Note also that the first ``negative'' element $\xi(1)$ is exactly the unique edge incident to $\nu$ in the initial matching $\Xscr^0$; denote it as $\psi(0)$.

As a result, we obtain $\psi(0)=\xi(1)<_\nu\psi(1)=\xi(2)<_\nu\cdots <_\nu\psi(p-1)=\xi(p)<_\nu \psi(p)$. It follows that for $i=1,\ldots,N$, the unique element of $\Ascr_\nu$ occurring in the matching $\Xscr^{i}$ is more preferred by $>_\nu$ than any other element of $\Ascr_\nu$ created at $i'$-th iteration with $i'<i$.

Now, arguing in a similar fashion, consider a vertex $\eta$ in the ``worker'' part $\Wscr$ with $\Ascr_\eta\ne\emptyset$. Let $\ell(1)<\ell(2)<\cdots<\ell(r)$ be the sequence of indexes $\ell\in\{1,\ldots,N\}$ such that $\Xscr^{\ell-1}_\eta \ne \Xscr^{\ell}_\eta$. Then for each $q=1,\ldots,r$, the difference between $\Xscr^{\ell(q)-1}_\eta$ and $\Xscr^{\ell(q)}_\eta$ consists of two edges, say, an edge $\xi'(q)$ in $\Xi^{\ell(q)}$ and an edge $\psi'(q)$ in $\Psi^{\ell(q)}$.

At $\ell(q)$-th iteration, the edge $\xi'(q)$ in the current matching $\Xscr$ is replaced by $\psi'(q)$. By rule~\refeq{lin_order}, $\psi'(q)$ becomes least preferred for $\eta$ in the current $\Ascr^{\ell(q)}_\eta$. Also the element $\xi'(1)$ coincides with the unique edge incident to $\eta$ in $\Xscr^0$, denoted as $\psi'(0)$.

We obtain $\psi'(0)=\xi'(1)>_\eta\psi'(1)=\xi'(2)>_\eta\cdots >_\eta\psi'(r-1)=\xi'(r)<_\eta \psi'(r)$. It follows that for $i=0,\ldots,N$, the unique element of $\Ascr_\eta$ occurring in the matching $\Xscr^{i}$ is more preferred than any other element of $\Ascr_\eta$ created at $i'$-th iteration with $i'>i$.

Using above observations, we finish the proof as follows. Suppose that some matching $\Xscr^i$ is unstable in $(\Vscr,\Ascr)$. Then there is an edge $\eps=\eta\nu\in \Ascr$ (with $\eta\in\Wscr$ and $\nu\in\Fscr$) blocking for $\Xscr^i$. Let $\eps$ be created at $i'$-the iteration, and let $\alpha$ and $\beta$ be the edges in $\Xscr^i$ incident to $\nu$ and $\eta$, respectively. Obviously, $\eps\notin \Xscr^i$ and $i'\ne i$.  But by reasonings above, if $i'<i$ then $\alpha>_\nu \eps$, while if $i'>i$ then $\beta>_\eta \eps$, contrary to the supposition that $\eps$ is blocking.
\end{proof}   
   
Next we explain how to obtain the required property for the remaining principal g-allocations. This is based on the observation that the obtained structure $\Ascr$, also denoted as $\Ascr[\Tscr]$, does not depend on the choice of full route $\Tscr$. 

More precisely, let us say that two full routes $\Tscr=(x^0,x^1,\ldots,x^N)$ and $\Tscr'=(z^0,z^1,\ldots,z^N)$ (with $x^0=z^0=\xmin$ and $x^N=z^N=\xmax$) are \emph{neighboring} if $x^i\ne z^i$ holds for exactly one $i\in\{0,\ldots,N\}$. Note that in this case there are two rotations $R$ and $R'$ applicable to $y:=x^{i-1}=z^{i-1}$ such that $x^i=y+\tau_R\chi^R$ and $z^i=y+\tau_{R'}\chi^{R'}$. (Here $\tau_R$ and $\tau_{R'}$ are the maximal feasible weights of $R$ and $R'$ at $y$, respectively. These $R$ and $R'$ commute at $y$ (cf.~\refeq{commute}), i.e. $x^{i+1}=z^{i+1}$ is obtained by applying $(R',\tau_{R'})$ to $x^i$, and $(R,\tau_R)$ to $z^i$.)

An important fact (cf.~\refeq{rot_prop}(iii)) is that the rotations $R$ and $R'$ are edge disjoint (though may intersect at vertices). We assert that for any two edges $e\in E_R$ and $e'\in E_{R'}$, their images (copies) in $\Xscr^i$ and $\Zscr^i$ are already vertex disjoint (where $\Xscr^i$ and $\Zscr'$ are the images of $x^i$ and $z^i$ in $\Ascr[\Tscr]$ and $\Ascr[\Tscr']$, respectively). Indeed, if $e\in R^-$ and $e'\in R'^-$, then their copies $\Xi(e)=h^{-1}(e)$ and $\Xi(e')=h^{-1}(e')$ (created at $i$-th iterations for $\Tscr$ and $\Tscr'$) lie in the matching $\Yscr$ ($=\Xscr^{i-1}=\Zscr^{i-1}$); therefore, $\Xi(e)$ and $\Xi(e')$ are vertex disjoint. And if $e\in R^+$, then by the construction, the set $\Psi(e)$ of copies of $e$ created at $i$-th iteration for $\Tscr$ connects only vertices occurring in $\Xi^i$, and similarly if $e'\in R'^+$, whence the assertion follows.

As a consequence, we obtain that the sets $\Ascr[\Tscr]$ and $\Ascr[\Tscr']$, along with the linear orders $\trleft^N$ and $>^N$ coincide. This implies that $\Zscr^i$ is a stable matching in $(\Vscr,\Ascr[\Tscr])$.

Since the lattice $(\Sscrpr,\precpr)$ is distributive, the full routes (viz. maximal chains) in it are connected via the neighboring relations as above, i.e. for any full routes $\Tscr$ and $\Tscr'$, there is a sequence $\Tscr=\Tscr_1,\Tscr_2,\ldots, \Tscr_K=\Tscr'$ of full routes such that any two $\Tscr_i,\Tscr_{i+1}$ are neighboring. Summing up the above reasonings, we obtain the following:
 \begin{numitem1} \label{eq:universalA}
for any principal stable g-allocation $x\in\Sscrpr$, its image $\Xscr$ in $(\Vscr,\Ascr)$ is a stable matching w.r.t. the linear preferences $>_\nu$ ($\nu\in\Vscr$), where $\Ascr=\Ascr[\Tscr]$ for a full route $\Tscr$, and for each $e\in E$, $\Xscr(e)$ consists of $x(e)$ first edges in $\Ascr(e)$ (w.r.t. $\trleft_e$).
  \end{numitem1} 

Now we extend the above construction to all stable g-allocations $x\in\Sscr$. By explanations in Sect.~\SSEC{poset_rot} (see~\refeq{gen_biject}), $x$ is associated with a closed function $\lambda=\phi(x)\in\Zset^\Pi$ on the rotational poset $(\Pi,\lessdot)$, which means that $\lambda$ does not exceed the weight function $\tau$, and for any $(R,\tau),(R',\tau')\in\Pi$, if $(R,\tau)\lessdot(R',\tau')$ and $\lambda(R',\tau')>0$, then $\lambda(R,\tau)=\tau$. The correspondence $x\mapsto\phi(x)$ gives a bijection between $\Sscr$ and the set of closed functions on $(\Pi,\lessdot)$, and such functions are interrelated to closed sets (or ideals) in $(\Pi,\lessdot)$. This can be expressed as follows, where $x\in\Sscr$ and $\lambda=\phi(x)$:
  \begin{numitem1} \label{eq:clos-clos}
\begin{itemize}
 \item[(i)] ~$x=\xmin+\sum(\lambda(R,\tau)\chi^R\colon (R,\tau)\in\Pi)$;
\item[(ii)] let $\lambda'$ be a function on $\Pi$ such that for each $(R,\tau)\in\Pi$, there holds $\lambda'(R,\tau)\in\{0,\tau\}$ if $0<\lambda(R,\tau)<\tau$, and $\lambda'(R,\tau)=\lambda(R,\tau)$ otherwise (i.e. $\lambda'$ is a sort of rounding of $\lambda$); then the function $\lambda'$ is closed as well, and $x':=\phi^{-1}(\lambda')$ is a principal g-allocation.
 \end{itemize}
  \end{numitem1}

Using this, consider $x\in\Sscr$ and $\lambda=\phi(x)$, and let $\Lambda$ be formed by all $(R,\tau)\in\Pi$ such that $0<\lambda(R,\tau)<\tau$. The principal g-allocation $x'$ closest to $x$ from below corresponds to the function $\lambda'$ on $\Pi$ such that $\lambda'(R,\tau)=0$ for $(R,\tau)\in\Lambda$, and $\lambda'(R,\tau)=\lambda(R,\tau)$ otherwise. Similarly, the principal g-allocation $x''$ closest to $x$ from above corresponds to $\lambda''$ such that $\lambda''(R,\tau)=\tau$ for $(R,\tau)\in\Lambda$, and $\lambda''(R,\tau)=\lambda(R,\tau)$ otherwise. Then
  $$
  x=x'+\sum(\lambda(R,\tau)\chi^R\colon (R,\tau)\in \Lambda) \;\;\mbox{and}
  \;\; x''=x+\sum(\tau-\lambda(R,\tau))\chi^R\colon (R,\tau)\in \Lambda).
  $$
Moreover  (in view of~\refeq{commute}), all rotations $R$ occurring in $\Lambda$ are different and pairwise edge-disjoint, and they are applicable to $x'$; let $\Rscr$ be the set of these, and for $(R,\tau)\in\Lambda$, denote the weights $\tau$ and $\lambda(R,\tau)$ by $\tau_R$ and $\lambda_R$, respectively.

Since $x''$ can be obtained from $x'$ by applying the rotations $R\in\Rscr$ with weights $\tau_R$ (in any order), there is a full route $\Tscr=(x^0,\ldots,x^N)$ containing both $x',x''$, say, $x'=x^i$ and $x''=x^j$. Then $j-i=|\Rscr|=:r$, and one can number the elements of $\Rscr$ as $R(1),\ldots,R(r)$ so that for $k=i+1,\ldots,j$, the principal g-allocation  $x^k$ is obtained from $x^{k-1}$ by applying $R(k)$ with weight $\tau_{R(k)}$. So the set $\Xscr^k$ is produced from $\Xscr^{k-1}$ by deleting $\tau_{R(k)}$ last elements related to negative edges in $R(k)$ and by adding $\tau_{R(k)}$ new elements related to positive ones. 
When, instead, we take only $\lambda_{R(k)}$ elements in these portions, over all $R(k)\in\Rscr$, we just obtain the corresponding matching $\Xscr$ in $(\Vscr,\Ascr)$ for $x$ in question. Keeping the linear preferences $\trleft_e$ for all $e\in E$, and $>_\nu$ for all $\nu\in\Vscr$, we can conclude with the following generalization of~\refeq{universalA}:
  \begin{numitem1} \label{eq:A-all_stable}
for any stable g-allocation $x\in\Sscr$, its image $\Xscr$ in $(\Vscr,\Ascr)$ is a stable matching w.r.t. the preferences $>_\nu$ ($\nu\in\Vscr)$, and for each $e\in E$, the set $\Xscr(e)$ consists of $x(e)$ first elements of $\Ascr(e)$.
  \end{numitem1}
  
Finally, let $V^0$ denote the set of vertices $v\in V$ with $\theta(v)=0$, and $E^<$ the set of edges $e\in E$ such that $|\Ascr(e)|<b(e)$. By the above construction, for each $v\in V-V^0$, the set $\Vscr_v$ ($=h^{-1}(v)$) consists of $\theta(v)$ vertices and all of them are covered by edges in $\Ascr$, whereas for each $v\in V^0$, the set $\Vscr_v$ consists of a single vertex (cf.~\refeq{split}) and this vertex is not covered by $\Ascr$. 

Now to complete the process, we fix one vertex in each $\Vscr_v$, denoted as $\omega(v)$, and extend the partial $\theta$-splitting $(\Vscr,\Ascr)$ to a required $\theta$-splitting $(\Vscr,\Escr)$ in a natural way, namely:
   \begin{numitem1} \label{eq:extendA}
for each edge $e=uv\in E^<$, connect the vertices $\omega(u)$ and $\omega(v)$ by $b(e)-|\Ascr(e)|$ parallel edges, making these edges less preferred for both $\omega(u)$ (compared with all edges in $\Ascr_{\omega(u)}$) and $\omega(v)$ (compared with all edges in $\Ascr_{\omega(v)}$), and inserting them at the end of $\trleft_e$ (thus extending the order $\trleft_e$ from $\Ascr(e)$ to $\Escr(e)$).
  \end{numitem1}

It is not difficult to conclude from~\refeq{extendA} that for each $x\in\Sscr$, the image $\Xscr=\Mscr_x$ of $x$ in $(\Vscr,\Ascr)$ constructed above continues to be a stable matching for $(\Vscr,\Escr)$, as required.

As to the last assertion in the theorem, it immediately follows from the observation that for $x,x'\in\Sscr$, if $x'$ immediately succeeds $x$ and is obtained from $x$ by applying a rotation $R$, then $h(R)$ is an ``augmenting'' cycle in $(G,>)$, whence $\Mscr_{x'}$ is a successor of $\Mscr_x$ (not necessarily an immediate one) in the lattice of stable matchings for $(G,>)$.  

This completes the proof of the theorem. \hfill$\qed\qed$


\section{Aggregated splitting} \label{sec:aggregate}

The above construction of $\theta$-splitting represented the stable g-allocations for a capacitated bipartite graph $G=(V,E)$ with integer choice functions via standard stable matchings in an incapacitated bipartite multigraph $\Gscr=(\Vscr,\Escr)$ (the ``model graph'') endowed with linear orders on (the stars of) its vertices. Here the size of $\Gscr$ may be large, as $\Gscr$ has $\sum(b(e)\colon e\in E)$ edges and at least $\sum(\theta(v)\colon v\in V)$ vertices.

In this section we try to decrease the size of representation, obtaining a ``smaller'' model graph with at most $2|E|+|V|$ vertices. This uses a reduction to stable allocations in a capacitated bipartite graph with quotas and linear orders on the vertices.

As before, let $\theta$ be as in hypotheses of Theorem~\ref{tm:split}. We will deal with a bipartite multigraph $\hat\Gscr=(\hat\Vscr=\hat\Wscr\sqcup\hat\Fscr,\hat\Escr)$ and a graph homomorphism $\hat h :\hat\Gscr\to G$ such that
  \begin{numitem1} \label{eq:hslash}
$\hat h$ sends ``workers'' to ``workers'', and ``firms'' to ``firms'', and for each edge $e=uv\in E$, the set ${\hat h}^{-1}(e)$ consists of edges between $\hat h^{-1}(u)$ and $\hat h^{-1}(v)$.
  \end{numitem1}
Besides, there are functions  $\hat b:\hat\Escr\to \Zset_+$ (of \emph{capacities}) and $\hat q:\hat\Vscr\to \Zset_+$ (of \emph{quotas}) subject to the following conditions:
  \begin{numitem1} \label{eq:cap-quota}
\begin{itemize}
 \item[(i)] for each $v\in V$, ~$\sum(\hat q(\nu)\colon \nu\in\hat h ^{-1}(v))=\max\{\theta(v),1\}$; and
 \item[(ii)] for each $e\in E$, ~$\sum(\hat b(\eps) \colon \eps\in \hat h^{-1}(e))=b(e)$.
 \end{itemize}
   \end{numitem1}

Also for each edge $e\in E$, there is a linear order $\hat\trleft_e$ on the edges in $\hat\Escr(e):=\hat h^{-1}(e)$, and for each vertex $\nu\in\hat\Vscr$, there is a linear order $\hat>_\nu$ on the set $\hat\Escr_\nu$ of edges incident to $\nu$. (Here for $\eps,\eps'\in\hat\Escr(e)$, ~$\eps\,\hat\trleft_e \,\eps'$ means that $\eps$ occurs in $\hat\Escr(e)$ earlier than $\eps'$, and for $\eps,\eps'\in \hat\Escr_\nu$, if $\eps\,\hat>_\nu \,\eps'$, then $\nu$ prefers $\eps$ to $\eps'$; cf. Definition~2 in Sect.~\SEC{split}.) We call $(\hat\Gscr,\hat h ,\hat b,\hat q,\hat\trleft,\hat>)$ (or briefly $\hat\Gscr=(\hat\Vscr,\hat\Escr)$) an \emph{aggregated} $\theta$-\emph{splitting} of $G$.

We are going to represent the stable g-allocations in $G$ via usual stable allocations in an aggregated $\theta$-splitting. This utilizes a partial $\theta$-splitting $(\Vscr,\Ascr)$ constructed in  the previous section. As before, we choose a full route $\Tscr=(\xmin=x^0,x^1,\ldots,x^N=\xmax)$ in which each $x^i$ is obtained from $x^{i-1}$ by applying a rotation $R_i$ with weight $\tau_i$. Let $\Ascr^0\subseteq\Ascr^1\subseteq\cdots \subseteq \Ascr^N=\Ascr$ be the sequence of edge sets formed in the procedure of constructing $(\Vscr,\Ascr)$. We iteratively transform this into a sequence of graphs (``partial aggregated $\theta$-splittings'') $\hat \Gscr^0=(\Vscr^0,\hat\Ascr^0), \hat\Gscr^1=(\Vscr^1,\hat\Ascr^1),\ldots,\hat\Gscr^N=(\Vscr^N,\hat\Ascr^N)$ (endowed with appropriate capacities, quotas and linear orders) as follows.

At iteration 0, for each $e\in E$ with $x^0(e)>0$, we identify (glue together) all edges in $\Ascr^0(e)$ (copies of $e$ in $\Ascr^0$) into one edge $\eps_e=\nu\eta$ and define 
  \begin{equation} \label{eq:iterat0} 
  b^0(\eps_e)=q^0(\nu)=q^0(\eta)=\varkappa^0(\eps_e):=x^0(e), 
  \end{equation}
using notation $b^0$ for capacities and $q^0$ for quotas in $\hat\Gscr^0$, and denoting by $\varkappa^0$ the allocation in $\hat\Gscr^0$ corresponding to $x^0$ (which will be automatically extended by zeros to the edges in $\hat\Ascr$ appeared on further iterations).  This results in the set $\hat\Ascr^0$ of pairwise vertex disjoint edges, and we denote the set of their endvertices by $\Vscr^0$. For $v\in V$, the set of copies (``preimages'') of $v$ in $\Vscr^0$ is denoted by $\Vscr^0_v$ (so each edge $e\in E_v$ with $x^0(e)>0$ contributes one element to $\Vscr^0_v$).
\smallskip

Next we describe $i$-th iteration for $i>0$. We assume that in its input there are: two graph homomorphisms $\hat h^{i-1}:\hat\Gscr^{i-1}\to G$  and $\hbar^{i-1}: \Gscr^{i-1}\to \hat\Gscr^{i-1}$; linear orders ${\hat\trleft}_e^{\,i-1}$ on the sets $\hat\Ascr^{i-1}(e)$  ($e\in E$), and $\hat >^{\,i-1}_\nu$ on the sets $\hat\Ascr^{i-1}_\nu$ ($\nu\in\Vscr^{i-1}$); capacities $ b^{i-1}$ on $\hat\Ascr^{i-1}$; quotas $q^{i-1}$ on $\Vscr^{i-1}$; and allocations $\varkappa^{i-1}$ on $\hat\Ascr^{i-1}$, which satisfy:
  \begin{numitem1} \label{eq:compos}
  \begin{itemize}
 \item[(i)]
the composition $\hat h^{i-1} \hbar^{i-1}$ coincides with $h^{i-1}$ (from $\Gscr^{i-1}$ to $G$);
\item[(ii)] 
the orders $\hat\trleft^{i-1}$ and $\trleft^{i-1}$ are agreeable, and similarly for $\hat>^{i-1}$ and $>^{i-1}$; formally: for $\eps,\eps'\in \Ascr^{i-1}(e)$ ($e\in E$), if $\eps\trleft^{i-1}_e \eps'$, then  $\hbar(\eps)\,\hat\trlefteq^{i-1}_e \hbar(\eps')$, and for $\eps,\eps'\in \Ascr^{i-1}_\nu$ ($\nu\in\Vscr$), if $\eps>^{i-1}_\nu\eps'$, then $\hbar(\eps)\, \hat\ge^{i-1}_{\hbar(\nu)}\, \hbar(\eps)$;
 \item[(iii)] $b^{i-1}$ is related to $\Ascr^{i-1}$, ~$ q^{i-1}$ to $\Vscr$, and $\varkappa^{i-1}$ to $x^{i-1}$; namely: $b^{i-1}(\eps)=|\hbar^{-1}(\eps)|$ for $\eps\in\hat\Ascr^{i-1}$, ~$q^{i-1}(\nu)=|\hbar^{-1}(\nu)|$ for $\nu\in\Vscr^{i-1}$, and $\sum(\varkappa(\eps)\colon \eps\in\hat h^{-1}(e))=x^{i-1}(e)$ for $e\in E$, where $\hbar$ stands for $\hbar^{i-1}$, and $\hat h$ for $\hat h^{i-1}$;
  \item[(iv)]  for $e\in E$ and $\eps,\eps'\in \hat\Ascr^{i-1}(e)$, if $\eps\,\hat\trleft^{i-1}_e \eps'$ and $\varkappa^{i-1}(\eps')>0$, then $\varkappa^{i-1}(\eps)=b^{i-1}(\eps)$.
 \end{itemize}
 \end{numitem1}
 
Take the rotation $R_i$. For a negative edge $c\in R^-_i$, consider the set $\Xi(c)$ of $\tau_i$ last edges in the subset $\Xscr^{i-1}(c)$ w.r.t. $\trleft^{i-1}_c$ (cf.~\refeq{Xcj}) and take its image $\hat\Xi(c):=\hbar^{i-1}(\Xi(c))$ in $\hat\Ascr^{i-1}(c)$. The next set $\Xscr^i(c)$ is obtained from $\Xscr^{i-1}(c)$ by removing $\Xi(c)$; therefore, at $i$-th iteration we should accordingly decrease the values of allocation $\varkappa$ on $\hat\Xi(c)$. More precisely, let $\xi_1,\ldots,\xi_k$ be the sequence of edges in $\hat\Xi(c)$ (following the order $\hat\trleft^{i-1}_c$). One can see that 
  \begin{numitem1} \label{eq:kappa-b}
  \begin{itemize}
\item[(i)]  $\Delta:=\varkappa^{i-1}(\xi_1)+\varkappa^{i-1}(\xi_2)\cdots+\varkappa^{i-1}(\xi_k)\ge\tau_i>  \varkappa^{i-1}(\xi_2)+\cdots \varkappa^{i-1}(\xi_k)$; and
 \item[(ii)]
$\varkappa^{i-1}(\xi_j)= b^{i-1}(\xi_j)$ for $j=1,\ldots,k-1$ ~(in view of~\refeq{compos}(iv)).
 \end{itemize}
 \end{numitem1}
Then we should assign 
  \begin{numitem1} \label{eq:xi_j}
~$\varkappa^i(\xi_j):=0$ for $j=2,\ldots,k$, and $\varkappa^i(\xi_1):=\varkappa^{i-1}(\xi_1)-(\Delta-\tau_i)$.
  \end{numitem1}
  
As a result, keeping the preimages of $c$, i.e. setting $\hat\Ascr^i(c):=\hat\Ascr^{i-1}(c)$, we obtain the desired equality $\sum(\varkappa^i(\eps)\colon \eps\in  \hat\Ascr^i(c))=x^i(c)$. At the same time, we preserve the capacities of all edges in $\hat\Xi(c)$; namely, $ b^i(\xi_j):=b^{i-1}(\xi_j)$ for all $j$.

Next, assuming that all negative edges in $R_i$ are already treated, we examine positive edges in $R_i$. Let $c,a,c'$ be consecutive edges in $R_i$, where $c,c'\in R^-_i$, ~$a=wf\in R^+_i$, and the vertices $w\in W$ and $f\in F$ are incident to $c$ and $c'$, respectively. The sets $\hat\Ascr^i(c),\hat\Ascr^i(c')$ and the function $\varkappa^i$ within the preimages of $c$ and $c'$ by $\hat h^i$ have already been assigned, and now our aim is to construct the set $\hat\Ascr^i(a)$ and function $\varkappa^i$ on the preimages of $a$ by $\hat h^i$. We use the set $\Psi(a)$ of (last) $\tau_i$ edges in $\Ascr^i(a)$ created by connecting the corresponding $\Wscr$-vertices in $\Xscr^{i-1}(c)$ and $\Fscr$-vertices in $\Xscr^{i-1}(c')$.

To construct the image $\hat\Psi(a)$ of $\Psi(a)$ by $\hbar^i$, consider the corresponding sequences $\hat\Xi(c)=(\xi_1,\ldots,\xi_k)$ (the image of $\Xi(c)$ by $\hbar^{i-1}$ ordered by $\hat\trleft^{i-1}_c$) and $\hat\Xi(c')=(\xi'_1,\ldots,\xi'_m)$ (the image of $\Xi(c')$ by $\hbar^{i-1}$ ordered by $\hat\trleft^{i-1}_{c'}$).

For $j\in [k]$, denote: the $\Wscr$-vertex of $\xi_j$ by $\nu_j$; the value $\varkappa^{i-1}(\xi_j)$ by $\varkappa_j$; the value $\varkappa_j+\cdots+\varkappa_k$ by $\alpha_j$; and the value $\varkappa_{j+1}+\cdots+\varkappa_k$ by $\beta_j$ (where $[k]$ denotes $\{1,\ldots,k\}$).

Similarly, for $r\in [m]$, denote: the $\Fscr$-vertex of $\xi'_r$ by $\eta_r$; the value $\varkappa^{i-1}(\xi'_r)$ by $\varkappa'_r$; the value $\varkappa'_r+\cdots+\varkappa'_m$ by $\alpha'_r$; and the value $\varkappa'_{r+1}+\cdots+\varkappa'_m$ by $\beta'_r$.

For $j\in[k]$ and $r\in[m]$, we assign new edge $\hat\psi_{jr}$ connecting $\nu_j$ and $\eta_r$ if and only if
  \begin{numitem1} \label{eq:jr-edge}
  \begin{itemize}
\item[(a)] ~$\Delta_{jr}:=\min\{\alpha_j,\alpha'_r\}-\max\{\beta_j,\beta'_r\}>0$, in case $j+r>2$;
 \item[(b)] ~$\Delta_{11}:=\tau_i-\max\{\beta_1,\beta'_1\}$.
  \end{itemize}
  \end{numitem1}
  
Let $P=P^i(a)$ be the set of such pairs $jr$. One can check that~\refeq{jr-edge} implies $\sum(\Delta_{jr}\colon jr\in P)=\tau_i$. Using this, we assign $\{\hat\psi_{jr}\colon jr\in P\}$ as the desired set $\hat\Psi(a)$, which is regarded as the image of $\Psi(a)$ by $\hbar^i$, and the preimage of $a$ by $\hat h^i$. Accordingly, for each $jr\in P$, we assign $\varkappa^i(\hat\psi_{jr}):=\Delta_{jr}$.

Note also that from~\refeq{jr-edge} it follows that for any $jr,j'r'\in P$, if $j<j'$ then $r\le r'$, and if $r<r'$ then $j\le j'$. This prompts the linear order on $\hat\Psi(a)$ defined by $\psi_{jr}\,\hat \trleft^i_a \,\psi_{j'r'}$ if $j+r<j'+r'$; this is agreeable with the linear order $\trleft^i_a$ on $\Psi(a)$. Like adding $\Psi(a)$ to $\Ascr^{i-1}(a)$, we should add $\hat\Psi(a)$ to the end of $\hat\Ascr^{i-1}(a)$, forming the next $\hat\Ascr^i(a)$ and $\hat\trleft^i_a$. 

Doing so for all positive edges $a$ of $R_i$, we eventually obtain the desired set $\hat\Ascr^i$, function $\varkappa^i$ on $\hat\Ascr^i$, linear orders $\hat\trleft^i_e$, $e\in E$, and maps $\hat h^i$ and $\hbar^i$, maintaining the corresponding properties in~\refeq{compos}. Emphasize that the above construction does not change the vertex set; namely, we have $\Vscr^0=\Vscr^1=\cdots=\Vscr^N$. 

Another important fact is that for each vertex $\nu\in\hat\Vscr$, the sum of values of $\varkappa^i$ over the set $\hat\Ascr^i_\nu$ of edges in $\hat\Ascr^i$ incident to $\nu$ is equal to the sum of values of $\varkappa^{i-1}$ over $\hat\Ascr^{i-1}_\nu$. This leads to preserving the quotas: $q^0(\nu)=q^1(\nu)=\cdots=q^N(\nu)$, yielding $\sum(\varkappa^i(\eps)\colon \eps\in\hat\Ascr^i_\nu)= q^N(\nu)$ and $\sum(q^N(\nu)\colon \nu\in \Vscr^N,\, \hat h^N(\nu)=v)=\theta(v)$ for all $v\in V$ with $\theta(v)>0$.

As to assigning the linear orders $\hat >^i_\nu$ on $\hat\Ascr^i_\nu$, $\nu\in\hat\Vscr$, we follow the rule similar to that concerning $>^i$ on $\Ascr^i$. Namely, if $\hat\Ascr^i_\nu$ acquires a new edge $\hat\psi$ (concerning some $a\in R^+_i$), we make $\hat\psi$ less (resp. more) preferred compared with any old edge contained in $\hat\Ascr^{i-1}_\nu$ when $\nu$ is in $\Wscr$ (resp. $\Fscr$). And the preferences by $\hat >^i_\nu$ between new edges $\eps,\eps'\in\hat\Ascr^i(e)$ sharing the same vertex $\nu$ depends on their orders by $\hat\trleft^i_a$; namely, in case $\eps\, \hat\trleft^i_a\, \eps'$, we assign $\eps\, \hat >^i_\nu\, \eps'$ when $\nu\in \Wscr$, and $\eps'\,\hat >^i_\nu\, \eps$ when $\nu\in \Fscr$.

Upon finishing the last, $N$-th, iteration of the above process, we obtain the desired partial $\theta$-splitting $(\hat \Gscr^N=(\hat \Vscr,\hat\Ascr^N),\hat h^N,b^N,q^N, \hat\trleft^N, \hat>^N)$. Extending each intermediate function $\varkappa^i$ ($0\le i\le N$) by zeros on $\hat\Ascr^N-\hat\Ascr^i$ (where $\hat\Ascr^i\subseteq\hat\Ascr^N$), one can conclude with the following: for $i=0,\ldots,N$,
  \begin{numitem1} \label{eq:lastN}
  \begin{itemize}
\item[(i)]  $\varkappa^i$ is a stable allocation in the multigraph $\hat\Gscr^N$ endowed with the capacities $b^N$, quotas $q^N$ and preference orders $\hat >^N$;
  \item[(ii)] for each edge $e\in E$, the sum of values $\varkappa^i(\eps)$ over the set
of preimages $\eps$ of $e$ by $\hat h^N$ (i.e. over $\hat\Ascr^N(e)$) is equal to $x^i(e)$, and the following ``saturation'' property on capacities is valid: for $\eps,\eps'\in\hat\Ascr^N(e)$, if $\eps\, \hat\trleft^N_e\, \eps'$ and $\varkappa^N(\eps')>0$, then $\varkappa^N(\eps)=b^N(\eps)$.  
  \end{itemize}
  \end{numitem1}
  
Here a proof is based on the above explanations and utilizes ideas from the proof of Lemma~\ref{lm:xi-Xi} for the (non-aggregated) construction of partial $\theta$-splitting $(\Vscr,\Ascr^N)$. Next, using an approach involving ``commuting rotations'', one can extend property~\refeq{universalA} to all principal g-allocations in the aggregated case, and further, using ``interpolations'' as in the proof of~\refeq{A-all_stable}, one can extend the construction to the set $\Sscr$ of all g-allocation in $G$. This gives a representation of all stable $x\in\Sscr$ via corresponding stable allocations for $(\hat\Gscr^N,\hat h^N,b^N,q^N,\hat>^N)$, denoted as $(\hat h^N)^{-1}(x)$. We leave details to the reader.

Finally, to obtain a representation of $\Sscr$ in a ``full'' aggregated $\hat\Gscr=(\hat\Vscr,\hat\Escr)$, we should take into account the original capacities $b$ on $E$ and  the vertices in $V^0$ (where $\theta$ is zero). 

More precisely, for a vertex $v$ of $G$, if $v\in V-V^0$, then the set $\hat h^{-1}(v)$ of preimages of $v$ in $\hat\Vscr$ is naturally assigned to be $(\hat h^N)^{-1}(v)$. And each $v\in V^0$ (if exists) generates a single vertex $\nu$; then $\hat h^{-1}(v)=\{\nu\}$. This gives the set $\hat\Vscr:=\hat h^{-1}(V)$. The quotas $\hat q$ on $\hat\Vscr$ are assigned by extending $q^N$ by zeros on $\hat\Vscr-\Vscr^N$. 

As to edges $e=uv\in E$, two cases are possible. If $\hat\Ascr^N(e)\ne \emptyset$, then the set of preimages of $e$ in $\hat\Gscr$ preserves: $\hat\Escr(e):=\hat\Ascr^N(e)$, as well as the linear order on it, and we assign the capacities $\hat b(\eps)$ of edges $\eps\in\hat\Escr(e)$ by possibly increasing the capacity of the last edge (w.r.t. $\hat \trleft_e$) and preserving those in the rest, providing the equalities $\sum(\hat b(\eps)\colon \eps\in\hat\Escr(e))=b(e)$ and $\hat b(\eps):=b^N(\eps)$ if $\eps$ is not the last in $\hat\Escr(e)$ (taking into account that, by the construction, $\sum(b^N(\eps)\colon \eps\in\hat\Ascr^N(e))$ is equal to the maximum value of $x(e)$ over $x\in\Sscrpr$). And if $\hat\Ascr^N(e)= \emptyset$, then $e$ generates a single vertex $\eps$ connecting (arbitrarily chosen) a vertex in $\hat h^{-1}(u)$ with a vertex in $\hat h^{-1}(v)$, which is assigned with the least (current) preference for each of these vertices. 

Analyzing the transformation of $(\hat\Gscr^N,\hat h^N,b^N,q^N,\hat>^N)$ into $(\hat\Gscr,\hat h,\hat b,\hat q,\hat>)$, one can realize that for each stable $x\in\Sscr$, its image $\varkappa:=(\hat h^N)^{-1}(x)$ in the former (partial) $\theta$-splitting, extended by zeros on $\hat\Escr-\hat\Ascr^N$, is stable in the latter one. (Here we also take into account that by the stability of $x$ no edge in $G$ connects two vertices in $V^0$.) Summing up the above reasonings, one can obtain
  \begin{theorem} \label{tm:aggregate_split}
Let $(\hat\Gscr=(\hat\Vscr,\hat\Escr),\hat h,\hat b,\hat q, \hat>)$ be the aggregated $\theta$-splitting constructed above. For $x\in\Zset^E_+$, define the function $\varkappa=\varkappa_x$ on $\hat\Escr$ so as to satisfy, for each $e\in E$, the equality $\sum(\varkappa(\eps)\colon \eps\in \hat\Escr(e))=x(e)$ and the ``capacity saturation'' condition similar to that in~\refeq{lastN}(ii), namely:
  \begin{itemize}
\item[\rm($\ast$)] for $\eps,\eps'\in\hat\Escr(e)$, if $\eps\,\hat\trleft_e\, \eps'$ and $\varkappa(\eps')>0$, then $\varkappa(\eps)=\hat b(\eps)$.
  \end{itemize}
Then the correspondence $x\mapsto \varkappa_x$, sends each stable g-allocation for $(G,b,C)$ to a stable allocation for $(\hat\Gscr,\hat b,\hat q, \hat>)$. Also the constructed multigraph $\hat\Gscr$ has at most $2|E|$ vertices (assuming that $G$ has no isolated vertices). Furthermore, if $x$ precedes $y$ in the lattice of stable g-allocations, then $\varkappa_x$ precedes $\varkappa_y$ in the lattice of stable allocations.
   \end{theorem}
 
Here the bound on $|\hat\Vscr|$ is seen from the observations that each edge of $G$ has at most one copy in the initial graph $\hat\Gscr^0$ and that each vertex from $V^0$ has a single copy in $\hat\Vscr$. Also we can roughly estimate the number of edges of $\hat\Gscr$ as $O(N^2|E|)$, where, as before, $N$ is the length of a full route for $(G,b,C)$ (which does not exceed $\bmax|E|$). To see this, for $i=0,\ldots,N$ and $e\in E$, let $m^i(e)$ denote the number of edges $\eps\in (\hat h^i)^{-1}(e)$ such that $\varkappa^i(\eps)>0$, and let $m^i:=\sum(m^i(e)\colon e\in E)$. In particular, $m^0\le|E|$. For $i\in[N]$, consider three consecutive edges $c,a,c'$ in the rotation $R_i$, where $c,c'$ are negative. At $i$-th iteration, the values of $\varkappa$ do not increase on the preimages of $c$ and $c'$ (cf.~\refeq{xi_j}) and increase on the new preimages of $a$, whence $\delta(d):=m^i(d)-m^{i-1}(d)$ is nonpositive for $d=c,c'$, and positive for $d=a$. Moreover, one can see from the construction (using~\refeq{jr-edge}) that $\delta(c)+\delta(c')+\delta(a)\le 1$. This implies that $m^i-m^{i-1}$ does not exceed $|R^-_i|$, or $|E|/2$. By induction on $i$, we obtain $m^i<i|E|$. Note also that by the construction each edge $\eps\in \hat\Ascr^N$ satisfies $\varkappa^i(\eps)>0$ for at least one $i$. Therefore, $|\hat\Ascr^N|\le m^0+m^1+\cdots+m^N< N^2|E|$, and the required estimate follows.

\end{document}